\documentclass[12pt,oneside,english]{amsart}
\usepackage[latin9]{inputenc}
\usepackage{geometry}
\usepackage{babel}
\usepackage{amstext}
\usepackage{amsthm}
\usepackage{amssymb}

\usepackage[unicode=true,pdfusetitle,
 bookmarks=true,bookmarksnumbered=false,bookmarksopen=false,
 breaklinks=false,pdfborder={0 0 1},backref=false,colorlinks=false]
 {hyperref}
\theoremstyle{definition}

\newtheorem{algorithm}{Algorithm}
\newtheorem{prop}{Proposition}

\theoremstyle{remark}

\numberwithin{equation}{section}

\usepackage{amsthm, amsmath, amssymb, amsfonts, algpseudocode, cite, multicol, anysize}

\newcommand{\ratl}{\mathbb{Q}}

\newcommand{\ints}{\mathbb{Z}}
\newcommand{\field}{\mathbb{F}}

\newcommand{\pideal}{\mathfrak{p}}

\newcommand{\ok}{\mathcal{O}_K}

\begin{document}

\title{Generating random factored ideals in number fields}


\author{Zachary Charles}
\address{University of Wisconsin-Madison, Dept. of Mathematics}
\curraddr{}
\email{zcharles@math.wisc.edu}
\thanks{The author would like to thank Eric Bach for his invaluable advice, suggestions, and encouragement. The author was supported by the National Science Foundation grant DMS-1502553.}


\date{}

\dedicatory{}

\begin{abstract}
We present a randomized polynomial-time algorithm to generate an ideal and its factorization uniformly at random in a given number field. We do this by generating a random integer and its factorization according to the distribution of norms of ideals at most $N$ in the given number field. Using this randomly generated norm, we can produce a random factored ideal in the ring of algebraic integers uniformly at random among ideals with norm up to $N$, in randomized polynomial time. We also present a variant of this algorithm for generating random factored ideals in function fields.
\end{abstract}

\maketitle

\section{Introduction}

We consider a generalization of the following problem: Given an integer $N > 0$, generate an integer in $[1,N]$ uniformly at random, along with its prime factorization, in polynomial time. Since there are currently no known polyonmial time factorization algorithms, we cannot simply generate an integer and factor it. Instead, we can generate the prime factorization uniformly at random.

In his thesis, Bach gave a randomized polynomial time method to uniformly produce a factored integer in $[\frac{N}{2},N]$~\cite{Bac1}. Bach's method uses only an expected number of $\log N$ primality tests in $[1,N]$~\cite{Bac2}. Since we can test for primality in polynomial time by the work of Agrawal, Kayal, and Saxena~\cite{AKS}, Bach's algorithm runs in randomized polynomial time. In 2003, Kalai gave another method for doing this using a conceptually simpler but slower method. Kalai's algorithm uses an expected number of $O(\log(N)^2)$ primality tests~\cite{Kal}. In 2016, Lebowitz-Lockard and Pomerance gave a variant of Kalai's algorithm to produce random factored elements on $\ints[i]$ with norm at most $N$~\cite{LLP}. This paper gives a generalization of these algorithms that will produce random factored ideals in $\mathcal{O}_K$ for any number field $K$. The method given will be polynomial in $N$ and the degree $d$, where we fix $d$ and assume $N$ tends to infinity.

We first discuss Kalai's algorithm and give a brief analysis. We then use these ideas to generate, in polynomial time, an integer in $[1,N]$ according to the distribution of norms of ideals in $\mathcal{O}_K$. We then discuss how to use this algorithm to produce random factored ideals of $\ok$ with norm up to $N$ uniformly at random using $O((\log N)^{d^2+d+1})$ primality tests. This algorithm is then modified to generate random factored ideals in function fields using $O((\log N)^{d^2+d+1})$ primality tests. While the algorithm described by Agrawal, Kayal, and Saxena runs in $O(log^{15/2} N)$, Lenstra and Pomerance later developed a primality testing algorithm with run time $O(\log^6 N)$ ~\cite{Len}. Using the latter algorithm, our algorithms for generating random factored ideals in number fields and function fields have expected run times that are $O((\log N)^{d^2+d+7})$.

\section{Kalai's Algorithm}

In his 2003 paper~\cite{Kal}, Kalai presents the following algorithm:

\begin{algorithm}

Input: A positive integer $N$.\\
Output: A positive integer $r \in [1, N]$ and its prime factorization, produced according to a uniform distribution.

1. Generate a sequence $N \geq s_1 \geq s_2 \geq \ldots \geq s_l = 1$ where $ s_1 \in \{1,\ldots, N\}, s_{i+1} \in \{1,\ldots, s_i\}$ are chosen uniformly at random. Stop when $s_j = 1$.\\
2. Let $r$ be the product of the prime $s_i$'s.\\
3. If $r \leq N$ output $r$ with probability $\frac{r}{N}$.\\
4. If $r > N$ or we do not accept $r \leq N$, return to step 1.\end{algorithm}

Let $p$ be a prime number. Then there is a first number $s_i$ produced in the range $\{1,2,\ldots, p\}$. It is chosen uniformly at random, so the probability that $s_i = p$ is $\frac{1}{p}$. Generating $e$ factors $p$ then occurs with probability $\frac{1}{p^e}(1-\frac{1}{p})$. The probability of outputting a given $r$ is then:
\begin{align*}
\frac{r}{N}Pr\bigg[r = \prod_{p \leq n} p^{v_p(r)}\bigg] &= \frac{r}{N}\prod_{p \leq N}\bigg(\frac{1}{p}\bigg)^{v_p(r)}\bigg(1-\frac{1}{p}\bigg)\\
&= \frac{r}{N}\frac{1}{r}\prod_{p \leq n}\bigg(1-\frac{1}{p}\bigg)\\
&= \frac{M_N}{N}
\end{align*}

Here $M_N = \prod_{p\leq N}(1-1/p)$. Therefore, if the algorithm terminates, it produces $r$ and its prime factorization with probability $\frac{1}{N}$. Moreover, the probability that Kalai's algorithm terminates after a single round is $M_N$. We therefore expect $M_N^{-1}$ trials before we output a number.

We now show that the expected number of primality tests is $O(\log(N)^2)$. Note that given some $s_i > 1$, the expected value of $s_{i+1}$ is $\frac{s_i}{2}$. Therefore, a given list $s_1, \ldots, s_l$ has expected length $O(\log N)$. We need an expected number of $O(\log N)$ primality tests for every round of the algorithm. Hence, we do an expected number of $O(M_N^{-1}\log N)$ primality tests overall. By Mertens' theorem, $M_N^{-1} = O(\log N)$. Therefore Kalai's algorithm uses $O(\log(N)^2)$ primality tests in expectation before the algorithm terminates.



\section{Generating Random Factored Norms}\label{gen_rand}

We now give a generalization of Kalai's algorithm that produces a random norm of an ideal in $\ok$, along with its prime factorization. Let $K$ be a number field of degree $d$, and $N$ an integer satisfying $d << \log(N)$. The algorithm will use a polynomial number of primality tests. In order to do so, we need knowledge of how the rational primes split. In general, this is computationally efficient for primes not dividing the discriminant of $K$ ~\cite{Mil1}. If $f(x)$ is the monic irreducible polynomial over $\ratl$ determining $K$, then we can determine how $p$ splits by factoring $f(x)$ modulo $\ints/p\ints$. Factoring polynomials over finite fields can be accomplished in randomized polynomial time~\cite{Knu}. The remaining prime numbers can be factored using a method of Chistov to factor polynomials in $\ratl_p$ in polynomial time~\cite{Chi}. We therefore assume we can factor rational primes in $K$.

For any integer $r > 0$, let $D(r)$ denote the number of ideals in $K$ of norm $r$. Let $p \in \ints$ be a rational prime with factorization $p\ok = \prod_{j=1}^{m} \pideal_j^{e_j}$. Let $N(\pideal_j) = p^{f_j}$. Finding $D(p^e)$ can now be reduced to determining the number of solutions to $\sum_{i=1}^m c_if_i = e$. This is an instance of the subset sum problem. We know that $e$ is $O(\log N)$. Therefore, we can solve the subset sum problem in $O(m\log N)$ operations~\cite{Cor}. Since $\sum_{i=1}^m e_if_i = d$, we know that $m \leq d$. The runtime for this becomes $O(\log N)$ as a result. Since $D(r)$ is multiplicative, we can calculate $D(r)$ for any $r$ relatively efficiently (certainly in polynomial time).

We will generate $r$ and keep $r$ with probability proportional to $\frac{rD(r)}{N}$. In order to guarantee a well-defined probability, we need to bound $D(r)$ by a factor that can be incorporated in to the selection of the factors. We use the following result. Let $\Omega(r)$ denote the number of prime factors of $r$ with multiplicity, and let $\Omega_d(r)$ be the number of prime factors (counting multiplicity) of $r$ that are greater than $d$.

\begin{prop}$D(r) \leq d^{\Omega(r)}$.\end{prop}
\begin{proof}Note that $D(r)$ is multiplicative. This is because prime ideals in $\ok$ have norms that are powers of prime numbers~\cite{Mil1}. Constructing an ideal of norm $r = \prod_p p^{e_p}$ amounts to constructing an ideal of norm $p^{e_p}$ for each prime $p | r$.

Say $p\ok = \prod_{i=1}^m \pideal_i^{e_i}$ for $\pideal \subset \ok$ prime. The norm is multiplicative, and since $N(p\ok) = p^d$, we must have $\sum_{i=1}^m N(\pideal_i)^{e_i} = d$. Fix $e$. Note that $D(p^e)$ will be at its largest when the $\pideal$ are all distinct (ie. $e_i = 1$) and all the $\pideal_i$ have norm $p$. Then $D(p^e)$ will be the number of unordered sets of $e$ elements taken from the $d$ prime ideals. This is bounded above by the number ordered sets, which is given by $d^e = d^{\Omega(r)}$.\end{proof}

We now present the main algorithm. For simplicity of analysis, we demonstrate the case where $d$ is even and $N \equiv d-1\bmod{d}$. Let $k \in \ints$ be such that $N = kd+(d-1)$. These assumptions are not necessary, but help simplify minor details.

\begin{algorithm}
Input: A positive integer $N$.\\
Output: A random integer in $[1,N]$. The integer is generated according to the distribution of norms of ideals in $\ok$ with norm up to $N$.

1. Generate $\lfloor\frac{d}{2}\rfloor$ lists of integers as follows. For $b \in \{1,3,5,\ldots, d-1\}$, generate a list $N \geq s_{1,b} \geq s_{2,b} \geq \ldots \geq s_{l_b,b} = 1$. We take $s_{1,b} \in \{1,d+b, 2d+b, \ldots, kd+b\}$, where $s_{1,b} = 1$ with probability $\dfrac{1}{kd+1}$ and is any other element with probability $\dfrac{d}{kd+1}$. Take $s_{i+1,b}$ in $\{1, d+b, 2d+b,\ldots, s_{i,b}\}$ as 1 with probability $\dfrac{1}{s_{i,b}-b+1}$ and any other element with probability $\dfrac{d}{s_{i,b}-b+1}$.\\
2. Let $r$ be the product of the prime $s_{i,b}$.\\
3. For all primes $p$ between 1 and $d$, do the following: Multiply $r$ by $p$ with probability $\dfrac{p-1}{p}$, and continue to multiply by $p$ with this probability until your first failure.\\
4. If $r \leq N$, keep $r$ with probability $M_d(r)\dfrac{\psi(r)D(r)}{d^{\Omega_d(r)}N}$.\\
5. If you did not keep $r$, go to step 1.\\
\end{algorithm}

If $d$ is odd, then we take $b \in \{1,3,5,\ldots, d-2\}$. This way we ensure that we are only picking odd numbers. If $N$ is some other value mod $d$, say $N = kd+j$ for $j < d-1$ then we can instead do the following. Let $N'$ be the smallest number above $N$ such that $N' \equiv d-1\bmod{d}$. Run steps 1-3 of the algorithm above with $N'$ substituted for $N$. Then run step 4 by rejecting all $r > N$ instead. For simplicity, we will analyze the algorithm in the case that $N \equiv d-1\bmod{d}$.

Let $\overline{n}$ denote the residue of $n\bmod{d}$. We define $\psi(r)$ as follows. Let $\psi(r)$ have the same factors as $r$, except that for any prime $p | r, p > d$, replace the factor of $p$ with a factor of $p-\overline{p}+1$. Note that $\psi(r) \leq r$ and they share the same prime factors, with multiplicity, for $2 \leq p < d$.

We define $M_d(r)$ by:
\begin{center}$\displaystyle M_d(r) = \prod_{2 \leq p < d}\alpha\bigg(\dfrac{1}{p-1}\bigg)^{v_p(r)}$\\
$\displaystyle \alpha = \binom{d+\lfloor\log N\rfloor-1}{\lfloor\log N\rfloor}^{-1}$\end{center}

\section{Analysis of the algorithm}

We wish to show that the probability of accepting $r$ is a well-defined probability. Let $g(r)$ be the product of all prime factors $p$ of $r$, with multiplicity satisfying $p > d$. Note that since $D(r)$ is multiplicative, we have
$$ D(r) = D(g(r)) \prod_{2 \leq p \leq d} D(p^{v_p(r)})$$

By the discussion in section \ref{gen_rand}, $D(p^{v_p(r)}) \leq \binom{d+v_p(r)-1}{v_p(r)}$. This follows from the fact that $D(p^{v_p(r)})$ is maximized when all the primes ideals lying above $p$ are all distinct and have norm $p$. Then $D(p^e)$ equals the number of unordered sets of size $e$ taken from the $d$ prime ideals lying above $p$.

Using the above we find
\begin{align*}
& M_d(r)\dfrac{\psi(r)D(r)}{d^{\Omega_d(r)}N}\\
&= \prod_{2 \leq p \leq d}\bigg[\alpha\bigg(\dfrac{1}{p-1}\bigg)^{v_p(r)}D(p^{v_p(r)})\bigg]\dfrac{\psi(r)}{N}\dfrac{D(g(r))}{d^{\Omega_d(r)}}\\
&\leq \prod_{2 \leq p \leq d}\bigg[\alpha D(p^{v_p(r)})\bigg]\dfrac{\psi(r)}{N}\dfrac{D(g(r))}{d^{\Omega_d(r)}}\\
&\leq \prod_{2 \leq p \leq d}\bigg[\binom{d+\lfloor\log N\rfloor-1}{\lfloor\log N\rfloor}^{-1}\binom{d+v_p(r)-1}{v_p(r)}\bigg]\dfrac{\psi(r)}{N}\dfrac{D(g(r))}{d^{\Omega(g(r))}}\\
&\leq 1
\end{align*}

Let $p > d$ be an odd prime with $p \equiv \overline{p}\bmod{d}$. Then we will produce exactly $e$ factors of $p$ with probability given by:
\begin{center}$\displaystyle \bigg(\dfrac{d}{p-\overline{p}+1}\bigg)^e\bigg(1-\dfrac{d}{p-\overline{p}+1}\bigg)$\end{center}

Let $r$ be some integer at most $N$. Recall that $\psi(r)$ is formed from $r$ by replacing all prime factors $p | r$, $p > d$ by $p-\overline{p}+1$. In particular, $r$ and $\psi(r)$ have the same prime divisors for $p \leq d$. Let $Pr\bigg[s = \prod_{d < p \leq N}p^{v_p(r)}\bigg]$ denote the probability that after step $2$ we have generated an integer $s$ that is the product of the prime factors of $r$ that are larger than $d$. The probability can be worked out as follows:
\begin{align*}
Pr\bigg[s = \prod_{d < p \leq N}p^{v_p(r)}\bigg] &= \prod_{d < p \leq N}\bigg(\dfrac{d}{p-\overline{p}+1}\bigg)^{v_p(r)}\bigg(1-\dfrac{d}{p-\overline{p}+1}\bigg)\\
&= \prod_{d < p \leq N}\bigg(\bigg(\dfrac{d}{p-\overline{p}+1}\bigg)^{v_p(r)}\bigg)\prod_{d < p \leq N}\bigg(1-\frac{d}{p-\overline{p}+1}\bigg)\\
&= d^{\Omega_d(r)}\prod_{d < p \leq N} \bigg(\frac{1}{p-\overline{p}+1}\bigg)^{v_p(r)}\prod_{d < p \leq N}\bigg(1-\frac{d}{p-\overline{p}+1}\bigg)\\
&= \dfrac{d^{\Omega_d(r)}}{\psi(r)}\prod_{2 \leq p \leq d} p^{v_p(r)}\prod_{d < p \leq N}\bigg(1-\frac{d}{p-\overline{p}+1}\bigg)
\end{align*}

Note that this last step used the fact that by definition of $\psi(r)$, we have
$$\prod_{d < p \leq N} \bigg(\frac{1}{p-\overline{p}+1}\bigg)^{v_p(r)} = \dfrac{\prod_{2 \leq p \leq d}p^{v_p(r)}}{\psi(r)}$$

Given any $r \leq N$, the probability that we generate $r$ after step 3 is then given by:
\begin{align*}
&\dfrac{d^{\Omega_d(r)}}{\psi(r)}\prod_{2 \leq p \leq d}\bigg[\bigg(\dfrac{p-1}{p}\bigg)^{v_p(r)}\bigg(1-\frac{p-1}{p}\bigg)\bigg]\prod_{2 \leq p \leq d}p^{v_p(r)}\prod_{d < p \leq N}\bigg(1-\frac{d}{p-\overline{p}+1}\bigg)\\
&= \dfrac{d^{\Omega_d(r)}}{\psi(r)}\prod_{2 \leq p \leq d}(p-1)^{v_p(r)}\frac{1}{p}\prod_{d < p \leq N}\bigg(1-\frac{d}{p-\overline{p}+1}\bigg)\end{align*}

Finally, the probability that we accept this $r$ is then given by:
\begin{align*}
& M_d(r)\dfrac{\psi(r)D(r)}{d^{\Omega_d(r)}N}\dfrac{d^{\Omega_d(r)}}{\psi(r)}\prod_{2 \leq p \leq d}(p-1)^{v_p(r)}\frac{1}{p}\prod_{d < p \leq N}\bigg(1-\frac{d}{p-\overline{p}+1}\bigg)\\
&= \dfrac{D(r)}{N}\prod_{2 \leq p \leq d}\alpha\bigg(\dfrac{1}{p-1}\bigg)^{v_p(r)}(p-1)^{v_p(r)}\frac{1}{p}\prod_{d < p \leq N}\bigg(1-\frac{d}{p-\overline{p}+1}\bigg)\\
&= \dfrac{D(r)}{N}\prod_{2 \leq p \leq d}\bigg(\frac{\alpha}{p}\bigg)\prod_{d < p \leq N}\bigg(1-\frac{d}{p-\overline{p}+1}\bigg)\end{align*}

Note that other than $D(r)$, all terms depend only on $d$ and $N$. Therefore, this generates a number with probability proportional to $D(r)$.

We now show that the algorithm above runs in polynomial time, with polynomial many primality tests and factorizations of rational primes. Summing over all $r$ at most $N$, the probability that we generate an ideal is
\begin{align*}
\dfrac{\sum_{r \leq N} D(r)}{N}\prod_{2 \leq p < d}\bigg(\frac{\alpha}{p}\bigg)\prod_{d < p \leq N}\bigg(1-\frac{d}{p-\overline{p}+1}\bigg)\end{align*}

Let $\displaystyle Z_n = \prod_{d < p \leq N}\bigg(1-\frac{d}{p-\overline{p}+1}\bigg)$. By the Wiener-Ikehara Tauberian theorem (see ~\cite{Ike} for reference), $\sum_{r\leq N} D(r)$ asymptotically approaches $C_KN$, where $C_K$ is the residue of the Dedekind zeta function of $K$ at $1$. Asymptotically then, the expected number of trials is $O(\alpha^{-d}Z_N^{-1}\prod_{2 \leq p < d} p)$.

By direct computation, we have:
\begin{align*}
\alpha^{-1} &= \binom{d+\lfloor\log N\rfloor-1}{\lfloor\log N\rfloor}\\
&= (d-1+\lfloor\log N\rfloor)(d-2+\lfloor\log N\rfloor)\ldots (1+\lfloor\log N\rfloor)\\
&\leq (2\log N)^{d-1}\end{align*}

We have $d$ factors of $\alpha^{-1}$, so this contributes $O(\log(N)^{d^2-d})$ to the expected number of trials. Note that the term $\prod_{2\leq p < d} p$ contributes a term that is bounded by a constant $d^d$, and is therefore $O(1)$ in terms of $N$. We now wish to find the contribution of the remaining term in the probability calculation above.

Simple estimates show:

\begin{align*}
Z_N &\geq \prod_{d < p \leq 2d} \bigg(1-\frac{d}{d+1}\bigg)\prod_{2d < p \leq N} \bigg(1-\frac{d}{p-d}\bigg)\\
&\geq c\prod_{2d < p \leq N}\bigg(1-\frac{2d}{p}\bigg)\end{align*}

Here $c = \prod_{d < p \leq 2d} \bigg(1-\frac{d}{d+1}\bigg) = (d+1)^{-d}$. By standard estimates, such as in~\cite{Ros}, we have:
\begin{align*}
\prod_{2d < p \leq N} \bigg(1-\frac{2d}{p}\bigg)^{-1} = O(\log(N)^{2d})\end{align*}

Hence we need $O(\log(N)^{d^2+d})$ trials before success. Since any given list has expected length $O(\log(N))$, this leads to $O(\log(N)^{d^2+d+1})$ primality tests.

We now use the above algorithm to generate random ideals of $\ok$, uniformly at random among all ideals with norm up to $N$. As previously stated, we assume that we for any rational prime $p$, we can find the factorization $p\ok = \prod_{i=1}^m \pideal_i^{e_i}$ with $N(\pideal_i) = p^{f_i}$. To find an ideal of norm $p^e$, it suffices to solve, as previously discussed, the subset sum problem:
\begin{center}$\displaystyle \sum_{i=1}^m c_if_i = e$\end{center}

Any solution to this corresponds to the ideal $\prod_{i=1}^m \pideal^{c_i}$. Since $e \leq \log N$ and $m \leq d$, we can clearly find all solutions in $O(\log^d N$ operations~\cite{Cor}. While this is not optimal, even the naive approach is dwarfed by the effort needed to generate the norm. After we find all solutions, we can then choose uniformly at random one of the solutions, which will give us one of the ideals of the desired norm. Therefore, the algorithm runs in a number of operations that is $O(\log^{d^2+d+1} N)$ primality tests and factorizations of $p\ok$. The primality testing is the dominant part of this in terms of run-time. Since we can perform primality tests in $O(\log^6 N)$, this gives us a run time that is $O(\log^{d^2+d+7}N)$ overall.

\section{Function Fields}

The analogy between number fields and function fields suggests that there should be an analogous algorithm for function fields. In particular, there are well-known randomized polynomial-time algorithms for factoring over $\field_q[t]$. To generate a factored random polynomial in $\field_q[t]$, we could simply generate one at random and then factor it in randomized polynomial time. Much more elegant ways exist that generate the factorization at random, the same idea used by the methods above.

Therefore, we would expect the ability to translate the algorithm above to arbitrary function fields. Fix a function field $K$ of degree $d$ and $N > 0$. We want to generate a random ideal $I \subset \ok$ with norm $r(t) \in \field_q[t]$ of degree at most $N$, along with the factorization of $I$. We will use the fact that we can perform primality testing over $\field_q[t]$ in polynomial time. We will also assume that for any irreducible polynomial $f(t) \in \field_q[t]$, we can factor $f(t)\ok$ in polynomial time. Since $\ok$ is a Dedekind domain, this holds for all $f(t)$ not dividing the discriminant, so there are only finitely many $f(t)$ that need to be factored as a one-time operation.

As in number fields, the main obstacle is generating the norm $g(t)$ of $I$ with probability proportional to the number of ideals with this norm. Once we can do this, then we can use our ability to factor $g(t)$ over $\ok$ and solve the corresponding subset sum problem to generate an ideal of $\ok$ uniformly at random with its factorization.

Let $g \in \field_q[t]$. Then we can consider $g$ to be a number written base $q$ by looking at its coefficients. Let $n(g)$ denote this number. Let $D(g)$ denote the number of ideals in $\ok$ with norm $g$. For any number $n \in \ints$, using its $q$-ary expansion, we can form a corresponding element $\beta(n) \in \field_q[t]$.

We present the following algorithm for generating a norm $g \in \field_q[t]$ with probabiity proportional to $D(g)$. It is virtually identical to the algorithm above, except that our concept of primality of a number $p$ is replaced by primality of the corresponding element $\beta(p) \in \field_q[t]$.

\begin{algorithm}
Input: A positive integer $N$.\\
Output: A random element in $F_q[t]$ with degree at most $N$. The polynomial is generated according to the distribution of norms of ideals in $\ok$ with degree of their norm at most $N$.

1. Generate $\lfloor\frac{d}{2}\rfloor$ lists of integers as follows. For $b \in \{1,3,5,\ldots, d-1\}$, generate a list $q^{N+1} > s_{1,b} \geq s_{2,b} \geq \ldots \geq s_{l_b,b} = 1$. We take $s_{1,b} \in \{1,d+b, 2d+b, \ldots, kd+b\}$, where $s_{1,b} = 1$ with probability $\dfrac{1}{kd+1}$ and is any other element with probability $\dfrac{d}{kd+1}$. Take $s_{i+1,b}$ in $\{1, d+b, 2d+b,\ldots, s_{i,b}\}$ as 1 with probability $\dfrac{1}{s_{i,b}-b+1}$ and any other element with probability $\dfrac{d}{s_{i,b}-b+1}$.\\
2. Let $r$ be the product of the $s_{i,b}$ such that $\beta(s_{i,b})$ is prime.\\
3. For all integers $p$ between 1 and $d$ such that $\beta(p)$ is prime, do the following: Multiply $r$ by $p$ with probability $\dfrac{p-1}{p}$, and continue to multiply by $p$ with this probability until your first failure.\\
4. If $r < q^{N+1}$, return $\beta(r)$ with probability $M_d(r)\dfrac{\psi(r)D(\beta(r))}{d^{\Omega_d(\alpha(r))}q^N}$.\\
5. If you did not return $\beta(r)$, go to step 1.\\
\end{algorithm}

Let $\overline{n}$ denote the residue of $n\bmod{d}$. We define $\psi(r)$ as follows. Factor $\beta(r)$ in to primes $g_i$. For each $g_i$ such that $n(g_i) > d$, replace the factor of $n(g_i)$ with $n(g_i)-\overline{n(g_i)}+1$. Note that $\psi(r) \leq r$ and $\beta(r), \beta(\psi(r))$ share the same factors $g_i$ such that $2 \leq n(g_i) \leq d$.

We define $M_d(r)$ by:
\begin{center}$\displaystyle M_d(r) = \prod_{\substack{f \in \field_q[t]\text{, f is prime}\\ 2 \leq n(f) < d}}\alpha\bigg(\dfrac{1}{n(f)-1}\bigg)^{v_{n(f)}(r)}$\\
$\displaystyle \alpha = \binom{d+N\lfloor\log q\rfloor-1}{N\lfloor\log q\rfloor}^{-1}$\end{center}

An almost identical argument to the one above shows that this produces $r$ with probability proportional to $D(\beta(r))$. Moreover, the algorithm produces a norm with probability

\begin{align*}
\dfrac{\sum_{g \in \field_q[t], \deg(g) \leq N} D(g)}{q^N}\prod_{\substack{g \in \field_q[t]\\ 2 \leq n(g) \leq d}} \dfrac{\alpha}{n(g)} \prod_{\substack{g\in\field_q[t]\\ d < n(g) < q^{N+1}}}\bigg(1- \dfrac{d}{n(g)-\overline{n(g)}+1}\bigg)\end{align*}

We first want to analyze $\sum_{g\in\field_q, \deg(g) \leq N} D(g)$. Let $\mathcal{D}^+_K$ denote the set of effective divisors of $V$, where $K$ is the function field of the projective variety $V$. Note that the zeta function for $K$ can be written as:
\begin{align*}
\zeta_K(s) &= \sum_{D \in \mathcal{D}^+_K} \dfrac{1}{q^{\deg(D)s}}\end{align*}

If we let $a_n$ denote the number of $g \in \field_q[t]$ of degree $n$, then this becomes:
\begin{align*}
\zeta_K(s) &= \sum_{n \geq 0} \dfrac{a_n}{q^{ns}}\end{align*}

Since $\zeta_K(s)$ converges absolutely for $Re(s) > 1$ and has a simple pole at $s = 1$, the analogue of the Wiener-Ikehara Tauberian theorem for function fields implies that:
\begin{align*}
\sum_{n \leq N} a_n = \Theta(q^N)\end{align*}

Let $Z_N$ be given by:
\begin{align*}
Z_N = \prod_{\substack{d < n(g) < q^{N+1}\\ g\text{ prime}}}\bigg(1-\dfrac{d}{n(g)-\overline{n(g)}+1}\bigg)\end{align*}

Note that the expected number of trials of the algorithm is:

$$O(\alpha^{-d}Z_N^{-1}\prod_{\substack{2 \leq n(g) \leq d\\ g\text{ prime}}} n(g))$$

By an analogous argument to the above, we have at most $d$ factors of $\alpha$ which contribute at most $(2N\log(q))^{d-1}$ to the expecte number of trials, while the last product in the estimate contributes an amount bounded by $d^d$. Therefore, it suffices to bound $Z_N^{-1}$. Using a version of the prime number theorem for $\field_q[x]$, one can prove in an analogous way to the proof above that $Z_N^{-1}$ contributes $O((N\log(q))^{2d})$ iterations in expected value.

Therefore, we require $O((N\log q)^{d^2+d})$ trials before success in expectation. Each of the $\frac{d}{2}$ lists has expected length $O(N\log q)$, so we require $O((N\log q)^{d^2+d+1})$ primality tests. Note that this gives us a randomized algorithm that is polynomial in the logarithm of the size of the input (as there are $q^N$ possible norms of the ideal to generate). Since primality testing can be performed in time that is $O((N\log q)^6)$, this gives us an overall runtime that is $O((N\log q)^{d^2+d+7})$.

\section{Further Work}

The algorithm above uses similar ideas to Kalai's algorithm. Bach's algorithm for generating factored integers runs in fewer primality tests, and so one could ask whether the ideas of Bach could be adapted to this setting in order to reduce the number of primality tests required. In general, one could determine ways to make the above algorithm run faster, in particular by reducing the number of primality tests required.

The algorithms above work for a fixed number or function field. In particular, their runtime is polynomial treating the degree of the field extension as constant. It remains an open question whether there is an algorithm for generating random factored ideals that runs in polynomial time, irrespective of the degree of the field. The methods above would have to be altered significantly to do so, due to their exponential dependence on the degree $d$.

There are other generalizations of this problem that could be considered. Due to the natural way in which principal ideals arise in number fields, one could ask for a variant of this algorithm that generates principal ideals uniformly at random. Clearly we could use the above algorithm to generate an ideal uniformly at random, and then only accept if it is principal. Since this occurs with probability $1/h$, where $h$ is the class number, this would result in a polynomial time algorithm provided we had a polynomial time way to recognize principal ideals. Unfortunately, this is a difficult question in arbitrary number fields, as there are no known polynomial time algorithms to detect whether an ideal in an arbitrary number field is principal.

\bibliographystyle{amsplain}
\bibliography{genfact}

\end{document}